\documentclass[11pt,letterpaper]{amsart}  
\usepackage{amssymb}
\usepackage{subcaption}
\usepackage{amsthm}
\usepackage{amsmath}
\usepackage{graphicx}
\usepackage{mathabx}
\usepackage{hyperref}
\usepackage{relsize}
\usepackage{pinlabel}
\usepackage{tikz-cd}
\usepackage[margin = 1.3in]{geometry}

\DeclareMathOperator{\im}{im}

\newtheorem{Thm}{Theorem}
\newtheorem{Prop}[Thm]{Proposition}
\newtheorem{Lem}[Thm]{Lemma}
\newtheorem{Cor}[Thm]{Corollary}

\theoremstyle{definition}
\newtheorem{Def}[Thm]{Definition}
\newtheorem{Prob}[Thm]{Problem}
\newtheorem{Ques}[Thm]{Question}

\theoremstyle{remark}
\newtheorem{Rem}[Thm]{Remark}
\newtheorem{Ex}[Thm]{Example}

\theoremstyle{definition}

\title{Surface Slices and Homology spheres}
\author{Clayton McDonald}
\address{Renyi Institute
}
\email[Clayton McDonald]{claytkm@gmail.com}

\begin{document}

\maketitle

\begin{abstract}
We develop the theory of the diagrammatics of surface cross sections to prove that there are an infinite number of homology 3-spheres smoothly embeddable in a homology 4-sphere but not in a homotopy 4-sphere. Our primary obstruction comes from work of Taubes.
\end{abstract}
\vspace{0.3 in}
\section{Introduction}
A long standing question in low-dimensional topology asks which 3-manifolds embed in which 4-manifolds, in particular the 4-sphere.

\begin{Prob}\cite[Problem 3.20]{kirbylist}
Under what conditions does a closed orientable 3-manifold smoothly imbed in $S^4$?
\end{Prob}

Budney and Burton tackled this problem by applying a variety of algebraic obstructions to the 11 tetrahedra census of manifolds \cite{BudneyBurton}. One more recent method for obtaining obstructions to 3-manifold embeddings has been to tweak existing gauge theoretic obstructions for studying the homology cobordism group  \cite{donald, Issamccoy}. 
However, all of these obstructions have generally been obstructions to embedding in homology spheres, rather than to embedding into the 4-sphere specifically. Understanding the difference between these two conditions is in some sense the hardest case of Kirby problem 3.20, as most of the existing obstructions to $S^4$ embeddings must vanish.

\begin{figure}
    \centering
    \includegraphics[width = 0.4\textwidth]{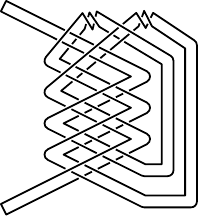}
    \caption{An example of a knot whose double branched cover embeds in a homology sphere but not a homotopy sphere.}
    \label{fig:final}
\end{figure}

\begin{Ques} \cite[Question 8.3]{BudneyBurton}
If a 3-manifold M admits a smooth embedding into a homotopy 4-sphere, does it admit a
smooth embedding (into) $S^4$? Are there 3-manifolds that embed in homology 4-spheres which do not embed in $S^4$?
\end{Ques}

In this paper, we establish a clear difference between these two properties, namely:

\begin{Thm}
\label{mainthm}
There exist an infinite number of integer homology 3-spheres that are embeddable in an integer homology 4-sphere but not in any homotopy 4-sphere.
\end{Thm}

For the purposes of this paper, we work entirely in the smooth category. The outline of the paper is as follows:

In Section 2, we give a gauge theoretic obstruction to embedding a manifold $Y$ in $S^4$, modeled after a strategy of Donald using a theorem of Taubes. The obstruction requires a definite cobordism from $Y$ to a specific 3-manifold, as well as a specific $\pi_1$ condition on the cobordism. 

In Section 3, we construct our examples via branched covers of cross sectional slices $K$ of knotted spheres $F$. The knotted spheres in question are spun knots. One of the main advantages to this perspective is that there is a natural inclusion of cyclic branched covers $\Sigma_n(S^3,K) \subset \Sigma_n(S^4,F)$. This fact has been used to great effect as an obstruction to double slicing of knots, but in the case of our proof, we will use it to describe families of 3-manifolds embeddable in homology 4-spheres. Moreover, our construction comes with a canonical cobordism to a standard manifold as required by the obstruction.

In Section 4, we verify the $\pi_1$ condition as well as the definiteness of our canonical cobordism, and show how our $\pi_1$ condition can be used to prove the main theorem. We also verify that we have infinitely many examples using a hyperbolic geometry argument. 


\section*{Acknowledgements.}  Thanks to my advisor Joshua Greene for his support. Thanks to Peter Feller, Tom Mark, Linh Troung, and Luya Wang for inspiring conversations during the AIM 4D Topology Workshop. Thanks to Antonio Alfieri for helpful comments in improving clarity. Thanks to Maggie Miller for pointing out an error in an earlier version of this paper. Thanks to Aliakbar Daemi and Chris Scaduto for clarifying the gauge theoretic side of things, and Tye Lidman for help with the hyperbolic geometry argument at the end. Thanks to Marco Marengon, Brendan Owens, and the anonymous referees for help with revision.

\section{The Gauge Theoretic Input} 



Our primary obstruction will be modeled on a corollary of a result of Taubes, attributed to Akbulut. For the purposes of this statement, a standard intersection form is one that is diagonalizable over the integers.

\begin{Thm}\cite[Proposition 1.7]{taubes}
\label{thm:taubes}
Let $\Sigma$ denote any homology sphere bounding a simply connected 4-manifold with nonstandard definite intersection form. Then there is no simply connected definite manifold with boundary $\Sigma \hash \overline{\Sigma}$.

\end{Thm} 
This theorem was recently generalized by Daemi \cite{daemi} to include a wider class of 3-manifolds (namely, any 3-manifold with nontrivial Fr\o yshov $h$-invariant \cite{Froyshov}). However, we will only use the Poincar\'e homology sphere for our proof, which has a well known nonstandard definite filling.

Using this theorem, Donald \cite{DonaldPhDThesis} noted in his thesis that one could obtain an obstruction to embedding into $S^4$:

\begin{Cor}\cite[Corollary 3.52]{DonaldPhDThesis}
Let $\Sigma$ be as above. If there exists a smooth definite simply-connected cobordism $W$ from $\Sigma \hash \overline{\Sigma}$ to a rational homology sphere $Y$, then $Y$ does not embed smoothly in $S^4$.
\end{Cor}
\begin{proof}[Idea of proof]
Assume $Y$ embeds in $S^4$, and let $X$ be one of the components of $S^4\setminus Y$. Then by Seifert-van Kampen, $\pi_1(Y)$ normally generates $\pi_1(X)$. This means that $ \pi_1(W) = 1$ normally generates  $\pi_1(W \cup_Y X)$ so $W \cup_Y X$ is also simply connected. Furthermore, $X$ is a rational homology ball, so $W \cup_Y X$ is definite, which by Theorem \ref{thm:taubes} is a contradiction.
\end{proof}


As Donald notes, it is easy to construct manifolds that satisfy the conditions of this obstruction. For example, if you attach 2-handles to a generating set for the fundamental group of $\Sigma \hash \overline{\Sigma}$, then this describes a simply connected cobordism, and with the proper choice of surgery coefficients, it can be made definite as well. The main challenge with using this obstruction is constructing a manifold that both embeds in a homology sphere and exhibits this definite cobordism. 

For our proof, we will use the same general framework, but with a few important changes to the construction. The obstruction to embedding in $S^4$ in the above case is really just an obstruction to $\pi_1(Y)$ normally generating $\pi_1(X)$ for $X$ some homology ball it bounds. (We note that this makes the above obstruction actually an obstruction to embedding into a homotopy sphere rather than $S^4$ in particular). Therefore,  one would not be able to find a manifold obstructable via this method via branched covers of slice knots, as both knots in $S^3$ and knotted surfaces in $B^4$ have fundamental groups normally generated by meridians.

For the purposes of our proof, we will slightly relax the $\pi_1$ condition on the cobordism. We note that the cobordism does not have to be simply connected, as long as we can guarantee that it becomes simply connected once we cap it off. 

Moreover, we will use a strengthening of the $\pi_1$ condition of Taubes' result (Theorem \ref{thm:taubes}). This strengthening follows from the main theorem in Taubes' paper similarly to Theorem \ref{thm:taubes}, but is not stated in this form.


\begin{Thm}\label{thm:su2}
    Let $\Sigma$ denote any homology sphere bounding a simply connected 4-manifold $X$ with nonstandard definite intersection form. Then any definite manifold $W$ with boundary $\Sigma \hash \overline{\Sigma}$ must have a non-trivial $\text{SU}(2)$ representation of its fundamental group.
\end{Thm}

\begin{proof}
    Assume that there exists a definite $W$ with boundary $\Sigma \hash \overline{\Sigma}$ admitting no nontrivial $\text{SU}(2)$ representation. By attaching a 3-handle, we obtain a definite self cobordism $C$ of $\Sigma$ admitting no nontrivial $\text{SU}(2)$ representation. If we attach successive copies of $C$ to $X$, we obtain an end-periodic and admissible 4-manifold $M$ in the sense of Taubes. Then by Theorem 1.4 of \cite{taubes}, we know that $M$ has diagonalizible intersection form, which we know is impossible as it contains the intersection form of $X$ as a sublattice.
\end{proof}

We note that the existence of an $\text{SU}(2)$ representation of the fundamental group will be crucial to the case analysis at the end of our argument. We will use Theorem \ref{thm:su2} in place of Theorem \ref{thm:taubes} as above along with a few facts about the representation theory of the Poincar\'e homology sphere's fundamental group.

 
\section{Surface slicing and cobordisms}
As stated in the introduction, our examples all come from branched covers of knots, and the knots are obtained using cross-sectional slices of knotted surfaces.

\begin{Def}
Let $F \subset S^4$ be a knotted surface, the image of an embedding $f: \Sigma_g \hookrightarrow S^4$ of a genus $g$ surface $\Sigma_g$. Then we say $K \subset S^3$ is a {\bf slice} of $F$ if $K$ can be realized as the intersection of $F$ and some equatorial $S^3$ in $S^4$.
\end{Def}

If we fix a knotted surface whose branched cover is a homology 4-sphere, then we can study its slices to create families of manifolds embeddable in the corresponding homology 4-sphere via branched covers.

\subsection{Slices and broken surface diagrams}

Here we will explore a new combinatorial technique for constructing surface slices.

The main diagrammatic input into our technique will be the {\bf broken surface diagram} of a knotted surface, a higher dimensional analogue of a knot diagram (See \cite{Carter1997KnottedSA} for a more in depth treatment of broken surface diagrams and decker sets). In the case of a knot diagram, we project the image of our embedding of $k:S^1 \hookrightarrow S^3$ to get an immersion $\overline{k}: S^1 \looparrowright S^2$, and recover the isotopy type of the knot using information about the points of self-intersection for the resulting immersed curve. Similarly, for a broken surface diagram we project the image of our embedding $F: \Sigma_g \hookrightarrow S^4$ to an immersion $\overline{F}: \Sigma_g \looparrowright S^3$, and record similar over/under information about the points of self intersection to recover the isotopy type.

To a given broken surface diagram we can also associate a {\bf decker set} diagram, a higher dimensional analogue of the chord diagram we would associate to a knot diagram. A {\bf chord diagram} is a pair $(K_0 = S^1,\{(a_1,b_1),(a_2,b_2),\dots (a_n,b_n)\})$ consisting of a circle $K_0$ with ordered pairs of points $(a_i,b_i)$ in $K_0$  corresponding to the exceptional preimages of $\overline{k}$, with the image of $a_i$ corresponding to the point that is higher in the projected coordinate. Correspondingly, a decker set is a pair $(F_0 = \Sigma_g,\{(\alpha_1,\beta_1),(\alpha_2,\beta_2),\dots (\alpha_n,\beta_n)\})$ consisting of a surface $F_0$ with ordered pairs of immersions $(\alpha_i,\beta_i)$ from $C = S^1$ or $A = [-1,1]$ to $F_0$ representing pairs of circles or pairs of arcs identified by the immersion $\overline{F}$, with $\alpha_i$ over $\beta_i$ in the projected coordinate. One cannot necessarily recover the broken surface diagram, and thus the isotopy class of the embedded surface in $S^4$, from the decker set, but it will be a useful visual aid for our constructions. Moreover, for orientable surfaces in $S^4$ we do not need any double arcs in our presentation \cite[Theorem 4.28]{Carter1997KnottedSA}.

This perspective for understanding knotted surfaces can be somewhat unwieldy, as diagrams of immersed surfaces are difficult to visualize. The key advantage in our case, however, is that broken surface diagrams naturally lie in $S^3$, so curves lying on $\im(\overline{F})$ can naturally be associated with knots. Moreover, we establish the following combinatorial criterion for when these curves actually represent slices of our surface, and the criterion only depends on the decker set. Therefore, we can discuss constructions of slices of knotted surfaces by merely considering curves on $F_0$, i.e.\ an abstract surface with certain decorated curves corresponding to the self intersections of the immersion $\overline{F}$.

\begin{Prop}\label{prop:combin}
Let $F:\Sigma_g \hookrightarrow S^4$ be a knotted surface, let $\overline{F}: \Sigma_g \looparrowright S^3$ (along with the associated crossing data) be a broken surface diagram representation of $F$ and let $(F_0,\{(\alpha_1,\beta_1),(\alpha_2,\beta_2),\dots (\alpha_n,\beta_n)\})$ be the corresponding decker set. Let $\gamma$ be a separating curve on $F_0$, and let $\mathcal{F}_1,\mathcal{F}_2$ be the components of its complement. If $\beta_i^{-1}(\mathcal{F}_1) \subset \alpha_i^{-1}(\mathcal{F}_1)$ for all $i$ or $\alpha_i^{-1}(\mathcal{F}_1) \subset \beta_i^{-1}(\mathcal{F}_1)$ for all $i$, then the curve $\overline{F}(\gamma)$ is a slice of $F$.
\end{Prop}
\begin{proof}
Assume without loss of generality that $\beta_i^{-1}(\mathcal{F}_1) \subset \alpha_{i}^{-1}(\mathcal{F}_1)$ for every $i$. To achieve a broken surface diagram projection, we remove two balls from $S^4$ and then project the resulting $S^3 \times [-1,1]$ to $S^3$. Our goal in this proof is then to reverse this projection process while maintaining $\gamma$ as a slice. Let us visualize $\im(\overline{F})$ as a subset of $S^3\times \{0\} \subset S^3 \times [-1,1]$. By abuse of notation, we will refer to this as $F_0$, with $\mathcal{F}_1$ and $\mathcal{F}_2$ referring to the appropriate subsets. Then we can push the interior of $\mathcal{F}_1$ up in the projected coordinate and the interior of $\mathcal{F}_2$ down in the projected coordinate so that $\mathcal{F}_1$ and $\mathcal{F}_2$ do not intersect in their interiors. 
Note that by the condition of the proposition, these pushes align with the crossing data of the original surface.


 We can then push the upper sheets of the self intersections of $\mathcal{F}_1$ (i.e. $\mathcal{F}_1 \cap im(\alpha_i)$) even higher in the projected coordinate in order to make them disjoint from their corresponding other sheets. Because $\beta_i^{-1}(\mathcal{F}_1) \subset \alpha_{i}^{-1}(\mathcal{F}_1)$, the lower sheet's intersection with the upper sheet is contained in the interior of $\mathcal{F}_1$, so we can raise the upper sheet above the lower sheet in the projected coordinate. This removes any intersections $\mathcal{F}_1$ has with itself.

We then apply a similar process to $\mathcal{F}_2$, pushing $\mathcal{F}_2 \cap im(\beta_i)$ down in the projected coordinate. We note that $\beta_i^{-1}(\mathcal{F}_1) \subset \alpha_i^{-1}(\mathcal{F}_1)$ for every $i$ implies that $\alpha_i^{-1}(\mathcal{F}_2) \subset \beta_i^{-1}(\mathcal{F}_2)$ for every $i$, so there is no obstruction to such a lowering removing any self intersections.

This results in an embedded surface $F' \subset S^3 \times [-1,1] \subset S^4$ whose intersection with $S^3\times \{0\}$ is $\gamma$ and has $\overline{F}$ as a broken surface diagram. Because broken surface diagrams capture the isotopy type of a surface, we know that $\gamma$ is a slice of $F$.
\end{proof}

\begin{figure}
    \centering
    \includegraphics[width=0.8\linewidth]{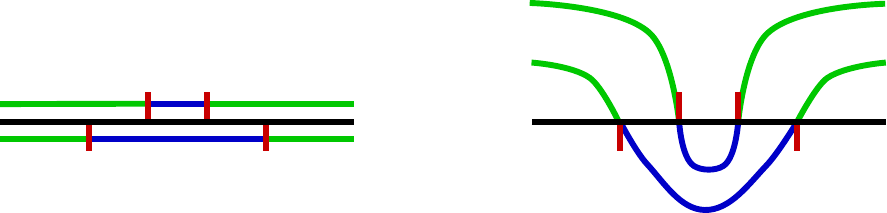}
    \caption{The above figure illustrates the coordinate pushing in Proposition \ref{prop:combin} restricted to an example double arc. In black we have the arc, with the colored arcs above and below representing the two preimages. The points in each preimage are mapped to the black arc via projection to the x coordinate. In red are the intersections of $\gamma$ with the double arc on the above arc and below arc respectively, and the green and blue components of the colored arcs represent $\mathcal{F}_1$ and $\mathcal{F}_2$ respectively. On the right picture we have an embedded version of the picture given by pushing the blue arcs up in the projected coordinate and pushing the green arcs down in the projected coordinate, fixing $\gamma$. We note that the resolved copy is not only embedded, but the prescribed over sheet is above the under sheet.}
    \label{fig:1ddecker}
\end{figure}

In order to understand how to use this proposition, we first need examples of knotted surfaces and decker sets for their broken surface diagrams.
\subsection{Spun knots}
\begin{Ex}\label{ex:spindecker}
For a knot $K \subset S^3$, we will define a knotted surface in $S^4$ called the {\bf spun knot} as follows. In general, we define the spin of a manifold $M$ to be $\partial(\overset{\circ}{M} \times D^2)$. Let $T_K \subset B^3$ be the two-ended tangle obtained from $K$ by removing a ball from $S^3$ containing a trivial arc of $K$. We can then define the spin of $K$ as $\partial(T_K\times D^2) \subset \partial(B^3\times D^2) = S^4$.
To project this to $S^3$, we can instead project each tangle ball to $D^2$, as $\partial(D^2\times D^2) \subset \partial(B^3\times D^2) = S^4$ defines an equatorial sphere in $S^4$. Because the projection can be applied equivariantly to each tangle ball, the spin of the knot diagram of $T_K$ gives a broken surface diagram for the spin of $K$. Moreover, we can spin the chord diagram of $T_K$ (a circle with marked points) to get the decker set diagram of the spin of $K$ (a sphere with marked curves as in Figure \ref{fig:chord}). We can get a variety of different broken surface diagrams for our spun knot by varying our choice of knot diagram for $T_K$.
\end{Ex}
\begin{figure}
    \centering
    \includegraphics[width = 0.8\textwidth]{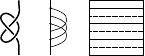}
    \caption{The tangle of the trefoil knot, its corresponding chord diagram (with the arrows pointing to the over crossing), and the decker set of the corresponding spun knot broken surface diagram projected so that latitudinal circles are mapped to horizontal lines. The dotted lines correspond to the lower sheets, with the solid lines being the upper sheets. The $\alpha_i$ and $\beta_i$ for the pairs of curves match up each point in a given latitude line with its corresponding point in the same longitudinal line, and the curves are matched up according to the chord diagram.}
    \label{fig:chord}
\end{figure}

By tracing through the chord diagram twice as in Figure \ref{fig:deckerdouble}, we get that $K\hash \overline{K}$ is a slice of the spin of $K$. In fact, the double of the natural ribbon disc for $K\hash \overline{K}$ is the spin of $K$. 

Using Proposition \ref{prop:combin}, we will show that the green curve on the left side of Figure \ref{fig:deckerdouble} represents a slice of the spin of $K$. Let us denote the interior of the green curve as $\mathcal{F}_1$. Note that $\mathcal{F}_1$ intersects every dotted curve $(\beta_i)$ in a thick region and every solid curve $(\alpha_i)$ in a thin region. This means that $\alpha_i^{-1}(\mathcal{F}_1) \subset \beta_i^{-1}(\mathcal{F}_1)$ for every $i$, so the condition of Proposition \ref{prop:combin} is satisfied (See Figure \ref{fig:1ddecker} for an explicit resolution of these double arcs).

To see that the green curve represents $K\hash \overline{K}$, first note that longitudinal/vertical lines in the decker set are copies of the projection of $T_K$ to the plane. Therefore, a rectangle surrounding the green curve should map into $S^3$ as the projection of $T_K$ times an interval $I$. If we slightly modify a vertical line segment so that it moves in the $I$ coordinate according to the over/under of the crossing data of $K$, we obtain a copy of $T_K$ embedded in our immersed surface. The green curve consists of two such perturbations of vertical line segments with opposite over/under data, connected by horizontal lines. This corresponds to connecting $T_K$ with $T_{\overline{K}}$ as in the right part of Figure \ref{fig:deckerdouble}. 

\begin{figure}
    \centering
    \includegraphics[width = 0.8\textwidth]{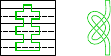}
    \caption{ $K\hash \overline{K}$ drawn on the decker set for the spun trefoil, and the corresponding knot diagram.
    }
    \label{fig:deckerdouble}
\end{figure}

\begin{Ex} One convenient feature of Proposition \ref{prop:combin} is that it only depends on the way $\mathcal{F}_1$ and $\mathcal{F}_2$ hit the marked curves, so we can arbitrarily modify parts of our slice outside the intersection locus to get new slices from old slices. 
The decker set we used in Example \ref{ex:spindecker} has a number of annuli in the complement of the marked curves, so we can apply a Dehn twist along any of these annuli to get new slices. 
If you start with the diagram of $K\hash \overline{K}$ in Figure \ref{fig:deckerdouble}, these Dehn twists act by applying a full twist if the twist annulus is the spin of an arc in the knot diagram that touches the unbounded region in the diagram (see Figure \ref{fig:deckersymm} and Figure \ref{fig:cylinder}).

\begin{figure}
    \centering
    \includegraphics[width = 0.8\textwidth]{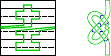}
    \caption{An even symmetric union over the trefoil drawn on the decker set for the spun trefoil, and the corresponding knot diagram in green. The blue arcs show the discs in the canonical ribbon disc for the symmetric union.}
    \label{fig:deckersymm}
\end{figure}
The knot $K\hash \overline{K}$ has a canonical ribbon disc given by $T_K \times I \subset B^3\times I$ (or $\overline{F}(\mathcal{F}_1)$). We can assign this disc a handle decomposition using a tangle diagram of $T_K$ by assigning a 0-handle (or disc) to each over crossing arc of the diagram, and a 1-handle (or band) to each arc connecting these arcs. In this way, each crossing of $T_K$ corresponds to a single ribbon singularity.

If we think of these Dehn twists with this canonical ribbon disc for $K\hash \overline{K}$ in mind, we can think of them as adding full twists to the bands given by our choice of diagram for $T_K$, as in Figure \ref{fig:cylinder}. Adding twists to the bands of the canonical ribbon disc for $K\hash \overline{K}$ gives a class of knots called {\bf symmetric unions} of $K$.
If every band is twisted for an even number of half twists, then we call this an {\bf even} symmetric union. With this in mind, our observations give the following result.
\end{Ex}
\begin{Prop}
Every even symmetric union of $K$ is a slice of the spin of $K$.
\end{Prop}

The diagram of $K \hash \overline{K}$ in Figure \ref{fig:deckerdouble} is more suggestive of its ribbon disc, but if we unfold the two connect summands, we get a more standard, crossing number minimizing example. If we do a similar unfolding process with the diagram of the symmetric union in Figure \ref{fig:deckersymm}, we get a diagram as in Figure \ref{fig:unfolded}, which is similar to the diagrams one would get using the definition of symmetric unions of \cite{Lamm2000SymmetricUA}. In this way, we can see the two definitions are equivalent, but our definition will be more suitable for the examples in this paper.

\begin{figure}
    \centering
    \includegraphics[width = 0.4\textwidth]{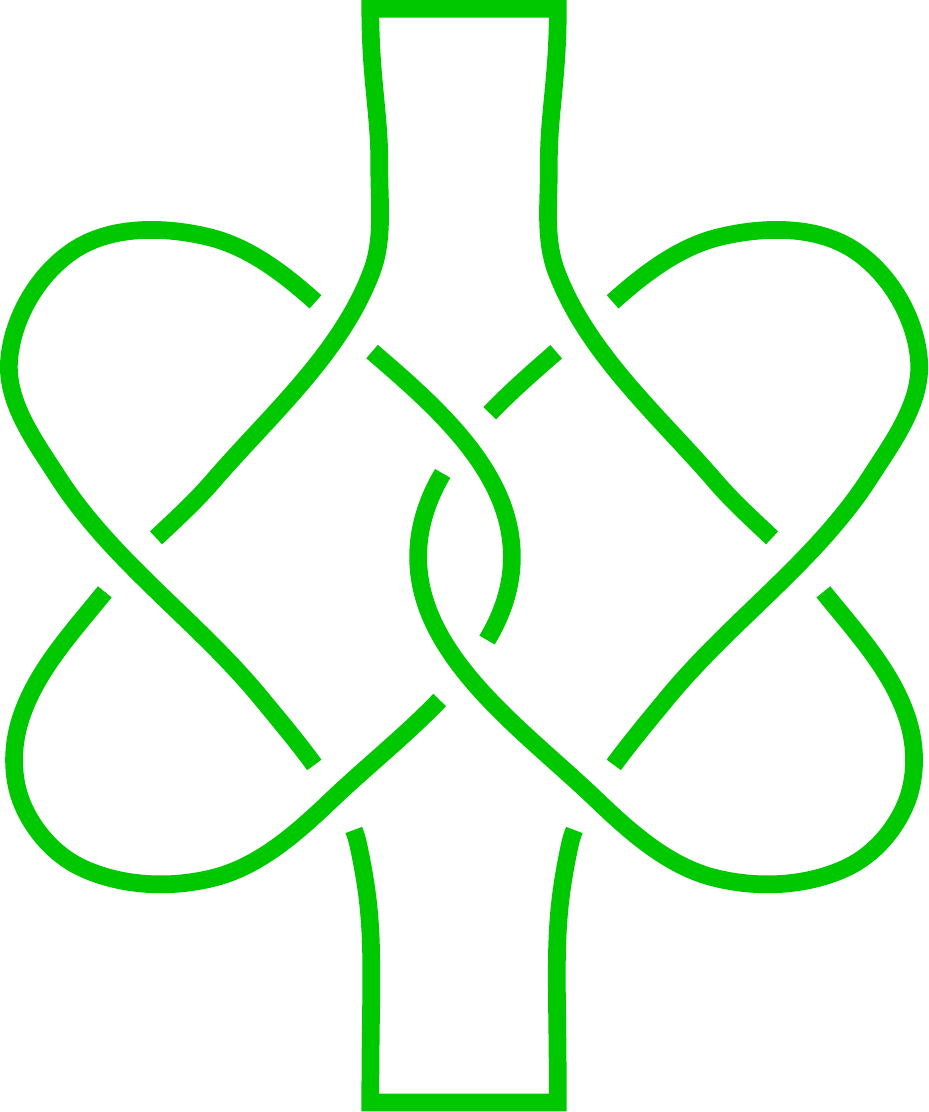}
    \caption{The unfolded symmetric union}
    \label{fig:unfolded}
\end{figure}

\subsection{Cobordisms}

The convenient thing about even symmetric unions is that we can realize these full twists to the bands by doing $\pm 1$-surgeries to meridians of the bands, as in Figure \ref{fig:cylinder}. Because the homotopy class of each of these surgery curves is a product of two meridians (namely, a meridian of $K$ and its corresponding meridian in $\overline{K}$), the curves lift to two curves in the double branched cover. Therefore, we can attach 2-handles along these curves on one side of $\Sigma_2(S^3, K\hash \overline{K})\times I$ to get us a cobordism between the double branched covers of $K\hash \overline{K}$ and our symmetric unions. Moreover, these surgeries preserve the property of being a homology 3-sphere, which will be seen in Lemma \ref{Lem:cob} (cf. \cite[Theorem 2.6]{Lamm2000SymmetricUA}). 

\begin{figure}
    \centering
    \labellist
    \small\hair 2pt
    \pinlabel $\simeq$ at 63 25
    \pinlabel $\simeq$ at 98 25
    \pinlabel -1 at 120 15
    \endlabellist
    \includegraphics[width = 0.8\textwidth]{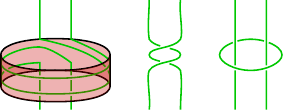}
    \caption{A local picture demonstrating how the Dehn twist around an annulus acts by applying a full twist, and how such a transformation is equivalent to $-1$ Dehn surgery on a meridian of the band.}
    \label{fig:cylinder}
\end{figure}

\section{The Proof}

\subsection{Definiteness of cobordisms}

\begin{Lem} \label{Lem:cob}
Assume $K$ is a knot whose double branched cover is an integer homology sphere. Then the intersection form of the cobordism $W$ associated to an even symmetric union on $K$ from its double branched cover to the double branched cover of $K\hash \overline{K}$ is represented by a diagonal matrix with $\pm 1$'s on the diagonal. In particular, the cobordism is definite if and only if all of the twist regions are of the same sign. Moreover, the double branched cover over these even symmetric unions is also an integer homology sphere.
\end{Lem}
\begin{proof}

The cobordism we construct is obtained from $S^3 \times [-1,1]$ by attaching $\pm 1$ framed 2-handles to $S^3 \times \{1\}$ and then taking a branched cover with branch set equal to $K \hash \overline{K} \times I$. The resulting branched cover can be similarly decomposed as $\Sigma_2(S^3,K \hash \overline{K}) \times I$ with 2-handles glued to the $\Sigma_2(S^3,K \hash \overline{K}) \times \{1\}$ end of it, where the framings and numbers of those handles are governed by the linking of the attaching circles with the branch set. Because the attaching circles are meridians of the cores of the bands, the linking number is 0, so each attaching curve will lift to two curves in the cover. This means that each 2-handle will lift to two 2-handles, and the framing will remain the same. Moreover, $\Sigma_2(S^3,K \hash \overline{K}) \times I$ has no second homology, so the homology is entirely generated by the cores of the 2-handles together with appropriate nullhomologous surfaces in $\Sigma_2(S^3,K \hash \overline{K}) \times I$.

To aid in the computation, we will cap off $S^3 \times \{-1\}$ with $B^4$ downstairs, extending the branch set to $B^4$ with the canonical ribbon disc for $K \hash \overline{K}$. We note that this does not add any second homology to the computation, as the double cover of this ball over the disc is a homology ball. However, it will allow us to see the generators of second homology more clearly. The meridians of the bands will be slice in the complement of the disc as after the band attachments there will be nothing in the way of the canonical slice disc for the unknot (refer to the right hand side of Figure \ref{fig:cylinder}). Therefore, we can cap off these framed 2-handle attachments with these discs in $B^4$. The framing of these resulting spheres is the same as the framing of the handle, they are pairwise disjoint, and they are disjoint from the branch set, so they will each lift to two pairwise disjoint spheres of the same framing. Therefore, the intersection form will be diagonal, with diagonal entries governed by the choice of twists as desired. Moreover, the lack of linking shows that each $\pm 1$ framed 2-handle attachment will preserve the homology of the 3-manifold, so both ends of the cobordism are homology spheres. Compare this proof to that of Cochran and Lickorish \cite{Cochran1986UnknottingIF}, noting that our situation is a bit simpler because the branch set has no singularities.
\end{proof}

\subsection{The fundamental group criterion}

\begin{Lem}
\label{lem:gp}
For an infinite number of the cobordisms associated to even symmetric unions over $K$, the fundamental group of the cobordism is $\pi_1(\Sigma_2(S^3,K))$. Moreover, infinitely many of them are definite.
\end{Lem}
\begin{proof} We obtain the cobordism associated to a given even symmetric union by attaching 
2-handles to $\Sigma_2(S^3,K\hash \overline{K})\times I$ corresponding to the appropriate twists in the bands. We can instead obtain the cobordism by attaching 2-handles to $(S^3\setminus \nu(K\hash \overline{K})) \times I$, taking the appropriate double cover and then reattaching the solid torus in each level set. If we can prove that the cobordism before the cover has fundamental group equal to $\pi_1(S^3\setminus K)$, then in the cover the cobordism will have fundamental group $\pi_1(\Sigma_2(S^3,K))$.

Fix a generating set $(m_1, \dots m_n)$ for $\pi_1(S^3\setminus K)$ consisting of meridians. The group $\pi_1(S^3\setminus (K\hash \overline{K}))$ is an amalgamated product of two copies of $\pi_1(S^3\setminus K)$, which is a free product with an added relation of identifying a fixed meridian in one copy with the corresponding meridian in the other copy (This can be seen by applying the Seifert-Van Kampen theorem to the decomposition with the splitting sphere as the intersection). We denote these two copies by $A_1 = \langle m_1, \dots, m_n \rangle$ and $A_2 = \langle m_1', \dots m_n'\rangle$, and say that $\pi_1(S^3\setminus (K\hash \overline{K})) = \langle m_1, \dots, m_n, m_1', \dots, m_n'| m_n = m_n'\rangle$ . 

The 2-handles added in the cobordism from $K\hash \overline{K}$'s complement to the symmetric union complement add a relation of identifying a meridian in $A_1$ with the corresponding meridian in $A_2$. Because knot groups are generated by meridians, we can add twists to identify each $m_i$ with its corresponding $m_i'$. After these attachments, $\pi_1$ of the cobordism is $\pi_1(S^3\setminus K)$, meaning that in the cover the cobordism will have fundamental group equal to $\pi_1(\Sigma_2(S^3,K))$. The framings of the 2-handles are irrelevant for this calculation, so we may assume they are all of the same sign, which by Lemma \ref{Lem:cob} means they can be made definite.
\end{proof}

\begin{Rem}
Note that the knot in Figure \ref{fig:final} is sufficent to satisfy this condition, as a full twist is added to each bridge arc in a bridge presentation of $T_{3,5}$.

\end{Rem}

\begin{proof}[Proof of Theorem \ref{mainthm}]

If we set $K = T_{3,5}$, then $\Sigma_2(S^3,K) = P$, the Poincar\'e homology sphere, which satisfies the necessary conditions for $\Sigma$ in Theorem \ref{thm:su2}. Moreover, the branched double cover of the spin of $K$ will be the spin of $P$ (i.e. $\partial((P\setminus B^3) \times D^2)$ = $(P\setminus B^3) \times S^1 \cup S^2 \times D^2$), which is a homology 4-sphere by Mayer-Vietoris. Therefore, it suffices for the proof of Theorem \ref{mainthm} to prove that for some family of even symmetric unions with base knot $T_{3,5}$ the associated cobordisms are definite, and that we can guarantee the resulting manifold cannot admit an appropriate $\text{SU}(2)$ representation after capping it off.

Let $J$ be an even symmetric union on $K$ such that the associated cobordism $W$ from $\Sigma_2(S^3,K\hash \overline{K})$ to $\Sigma_2(S^3,J)$ is definite and $\pi_1(W) = \pi_1(P)$ (See Figure \ref{fig:final} for example). Let $Y = \Sigma_2(S^3,J)$. Assume that $Y$ embeds smoothly into a homotopy 4-sphere $S$, and let $X_1, X_2$ be the two integer homology balls it bounds which together compose $S$. 

Because $X_1$ and $X_2$ are integer homology balls, they don't affect the definiteness of the intersection form, so it suffices to show for the proof that gluing one of the two $X_i$ to $W$ will result in a manifold with no appropriate $\text{SU}(2)$ representation. To do this, let $W_i$ be the space obtained by gluing $X_i$ to $W$ along $Y$, and let $T$ be the singular space obtained by gluing both $X_i$ to $W$ along $Y$ (see Figure \ref{fig:triad}). Let $G_i$ be the image of the inclusion induced map $\iota:\pi_1(W) \mapsto \pi_1(W_i)$.

\begin{figure}
    \centering
    \labellist
    \small\hair 2pt
    \pinlabel $X_1$ at 57 50
    \pinlabel $X_2$ at 57 15
    \pinlabel $W$ at 21 32
    \endlabellist
    \includegraphics[width = 0.4\textwidth]{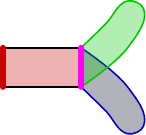}
    \caption{A schematic of the singular space $T$. The pink line segment represents $Y$ and the red segment represents $P \hash \overline{P}$.  }
    \label{fig:triad}
\end{figure}
Because $W$ is built from $Y\times I$ by attaching 2-handles, we know that $\pi_1(W)$ is a quotient of $\pi_1(Y)$.  Because the $X_i$ are submanifolds of $S$, we also know that $\pi_1(W_i)$ is normally generated by $G_i$, which is a quotient group of $\pi_1(P)$ by Lemma \ref{lem:gp}. Moreover, because the $X_i$ glue together to form $S$, we know that $\pi_1(T) = 1$, as it is constructed by gluing on thickened 2 cells to $S$. By Seifert-van Kampen, we then know that $\pi_1(W_i)$, $\pi_1(P)$, and $1$ fit into a pushout diagram as in Figure \ref{fig:pushout}. Moreover, we know that $\pi_1(P) = \text{SL}_2(\mathbb{F}_5)$, the group of 2 by 2 matrices over the field of 5 elements with determinant 1. 

\begin{figure}
    \centering
\begin{tikzcd}
\text{SL}_2(\mathbb{F}_5) \arrow[d,"f"'] \arrow[r,"g"] &
\pi_1(W_1) \arrow[d,"\beta"']  \\
\pi_1(W_2) \arrow[r,"\alpha"] &
1 \\
\end{tikzcd}
    \caption{The pushout diagram containing $\pi_1(W_1)$ and $\pi_1(W_2)$.}
    \label{fig:pushout}
\end{figure}

The group $\text{SL}_2(\mathbb{F}_5)$ has only the alternating group $A_5$ as a nontrivial quotient, so we can conclude our proof with a case analysis on $G_i$: 

{\bf Case 1: At least one $G_i = 1$}

If $G_i = 1$, then $\pi_1(W_i) = 1$, and $W_i$ is the definite cobordism required for our obstruction, as $\pi_1(W_i)$ admits no nontrivial representations.

{\bf Case 2: A least one $G_i = A_5$}

If $G_i = A_5$, then $W_i$ does not admit a nontrivial $\text{SU}(2)$ representation that extends nontrivially to the boundary. This is because $A_5$ does not admit a nontrivial $\text{SU}(2)$ representation (as it is a simple group and not the fundamental group of a spherical 3-manifold), and $G_i$ is the image of $\pi_1(\Sigma \hash \overline{\Sigma})$ in $W_i$.

{\bf Case 3: $G_i = \text{SL}_2(\mathbb{F}_5)$ for both $i$}

The last possible case is that $G_1 = G_2 = \text{SL}_2(\mathbb{F}_5)$. In this case, both of our initial maps ($f$ and $g$) in our pushout are injective. In this case, the associated free amalgamated product at the end of the commutative diagram cannot be the trivial group, as $\alpha$ and $\beta$ must be injective as well by the normal form theorem for free products with amalgamation \cite[p.187]{LyndonSchupp}. The resulting contradiction completes the case analysis.

To truly complete the proof, we need to show that there are infinitely many distinct manifolds among the infinite number of different branched covers of twist spins with definite associated cobordism. For fixed twist regions, we notice that all even symmetric unions as in the left of Figure \ref{fig:infmany} can be changed to the unknot via the same four rational tangle replacements shown on the right of Figure \ref{fig:infmany}, with the only change being the induced framings on two of the tangle replacements. This means that via the Montesinos trick, we can write them all as surgeries on the same 4-component link, where the framings of two of the components vary according to the number of twists. Namely, two of the components are 0-framed, and the other two are $2n$ and $2m-1$, where $n$ and $m$ correspond to the number of full twists in the symmetric union diagram. We note that as long as both $n$ and $m$ are of the same sign, this surgery diagram is one of a manifold that embeds in a homology sphere but not a homotopy sphere.

\begin{figure} [h]
\begin{subfigure}{.5\textwidth}
  \centering
  \includegraphics[width=.8\linewidth]{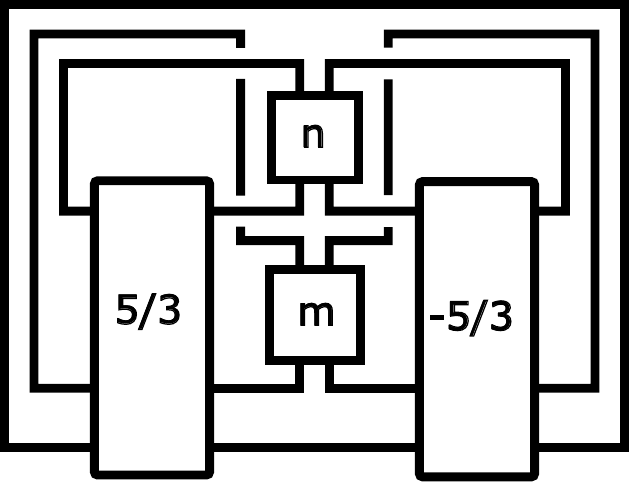}
  \caption{}
  \label{fig:sfig1}
\end{subfigure}%
\begin{subfigure}{.5\textwidth}
  \centering
  \includegraphics[width=.8\linewidth]{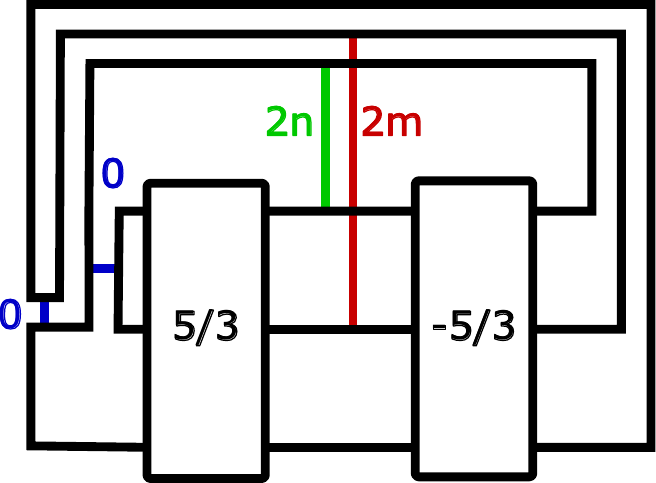}
  \caption{}
  \label{fig:sfig2}
\end{subfigure}
\caption{Left: A symmetric union of $T_{3,5}$ . The numbers in the boxes represent full twists. Right: The four band attachments in blue, red and green that give the symmetric union from the unknot. The color coded numbers represent the difference of the framings of the bands in half twist from the blackboard framing. To realize the double cover of the knot as a surgery on a 4-component link, take the double cover on the black unknot, and the core arcs of the bands will all lift to knots, with framings prescribed by the framings of the arcs.}
\label{fig:infmany}
\end{figure}

From here, we use Snappy \cite{SnapPy} and Regina \cite{regina} to verify that the complement of the 4th component is hyperbolic if we do $0,0$ and $2$ surgeries to the other components. Because we can range over all positive integers for $m$, we know by a result of Neumann and Zagier \cite[Theorem 1A]{MR815482} that the hyperbolic volume of some subsequence of these surgeries will approach the volume of the cusped manifold monotonically. This means that for some subsequence of $m$'s, the manifolds will all be hyperbolic and of distinct volume, meaning they are all distinct manifolds. This completes the proof.
\end{proof}

\bibliographystyle{plain}
\bibliography{slicehom}

\end{document}